\documentclass[11pt,twoside]{article}

\usepackage{amssymb}
\usepackage{amsmath}
\usepackage{amsthm}
\usepackage{color}

\allowdisplaybreaks

\def\rr{{\mathbb R}}

\def\nn{{\mathbb N}}

\def\cm{{\mathcal M}}

\def\fz{\infty}

\def\supp{{\mathop\mathrm{\,supp\,}}}

\def\lz{\lambda}

\def\ez{\epsilon}

\def\ls{\lesssim}
\def\gs{\gtrsim}

\def\tbz{{\triangle_\lz}}
\def\dmz{{dm_\lz}}

\def\plz{{P^{[\lz]}_t}}

\def\rrp{{{\mathbb\rr}_+}}
\def\bmoz{{{\rm BMO}(\mathbb\rrp,\, dm_\lz)}}

\def\ltz{{L^2(\rrp,\, dm_\lz)}}
\def\lrz{{L^r(\rrp,\, dm_\lz)}}

\def\loz{{L^1(\rrp,\, dm_\lz)}}

\def\lpz{{L^p(\rrp,\, dm_\lz)}}

\def\linz{{L^\fz(\rrp,\, dm_\lz)}}
\def\hoz{{H^1({\mathbb R}_+,\, dm_\lz)}}

\def\varpz{{\mathcal V}_\rho(P^{[\lz]}_\ast)}
\def\oscpz{{\mathcal O(P^{[\lz]}_\ast)}}
\def\boscpz{{\mathcal O\lf(P^{[\lz]}_\ast\r)}}

\def\dsum{\displaystyle\sum}
\def\dint{\displaystyle\int}

\def\dfrac{\displaystyle\frac}
\def\dsup{\displaystyle\sup}

\def\r{\right}
\def\lf{\left}

\def\beeqn{\begin{equation}}
\def\eneqn{\end{equation}}
\def\beeqns{\begin{equation*}}
\def\eneqns{\end{equation*}}
\def\beeqa{\begin{eqnarray}}
\def\eneqa{\begin{eqnarray}}
\def\beeqas{\begin{eqnarray*}}
\def\eneqas{\begin{eqnarray*}}
\def\besp{\begin{split}}
\def\ensp{\begin{split}}
\def\noz{\nonumber}

\newtheorem{thm}{Theorem}[section]
\newtheorem{lem}[thm]{Lemma}%[section]
\newtheorem{prop}[thm]{Proposition}%[section]
\newtheorem{rem}[thm]{Remark}%[section]
\newtheorem{defn}[thm]{Definition}%[section]
\numberwithin{equation}{section}

\begin{document}

\arraycolsep=1pt

\title{\Large\bf  Oscillation and  variation for semigroups  associated with Bessel operators}
\author{Huoxiong Wu, Dongyong Yang\,\footnote{Corresponding author}\, and Jing Zhang}

%\medskip
\date{}
\maketitle

\begin{center}
\begin{minipage}{13.5cm}\small

{\noindent  {\bf Abstract:}\  Let $\lambda>0$  and
$\triangle_\lambda:=-\frac{d^2}{dx^2}-\frac{2\lambda}{x}
\frac d{dx}$ be the Bessel operator on $\mathbb R_+:=(0,\infty)$. We
show that the oscillation operator ${\mathcal O(P^{[\lambda]}_\ast)}$ and variation operator
${\mathcal V}_\rho(P^{[\lambda]}_\ast)$ of the Poisson semigroup $\{P^{[\lambda]}_t\}_{t>0}$ associated with
 $\Delta_\lambda$ are both bounded on $L^p(\mathbb R_+, dm_\lambda)$ for $p\in(1, \infty)$, $BMO({{\mathbb R}_+},dm_\lambda)$,
 from $L^1({{\mathbb R}_+},dm_\lambda)$ to $L^{1,\,\infty}({{\mathbb R}_+},dm_\lambda)$,
 and from $H^1({{\mathbb R}_+},dm_\lambda)$ to $L^1({{\mathbb R}_+},dm_\lambda)$, where  $\rho\in(2, \infty)$
 and $dm_\lambda(x):=x^{2\lambda}\,dx$.
As an application, an equivalent characterization of
$H^1({{\mathbb R}_+},dm_\lambda)$ in terms of ${\mathcal V}_\rho(P^{[\lambda]}_\ast)$ is also established.
All these results hold if $\{P^{[\lambda]}_t\}_{t>0}$ is replaced by the heat semigroup $\{W^{[\lambda]}_t\}_{t>0}$. }

\end{minipage}
\end{center}

\bigskip
\bigskip

{ {\it Keywords}: oscillation; variation; Bessel operator; Poisson semigroup; heat semigroup.}

\medskip

{{Mathematics Subject Classification 2010:} {42B20; 42B35; 42B30}}

\section{Introduction and statement of main results\label{s1}}

Let $(\mathcal X,  \mu)$ be a measure space and ${\mathcal T}_{\ast}:=\{T_\epsilon\}_{\epsilon>0}$
 a family of operators bounded on $L^p(\mathcal X,  \mu)$ for $p\in(1, \fz)$
such that $\lim_{\epsilon\to0}T_{\epsilon}f$ exists
in some sense.  A classical way to measure the speed of convergence of $\{T_\epsilon\}_{\epsilon>0}$ is to
study square functions of the type $(\sum_{i=1}^\fz|T_{\epsilon_i}f-T_{\epsilon_{i+1}}f|^2)^{1/2}$,
where $\epsilon_i\to0$. Recently, other expressions have been considered,
among which are the $\rho$-variation and the oscillation operators;
see, for instance, \cite{Bou,bct,b14,cjrw1,cjrw2,cmtt,cmtv,gt,jr,jsw}. Recall that variation operator $\mathcal{V}_{\rho}({\mathcal T}_{\ast}f)$ is defined by
 \begin{equation}\label{d-var}
 \mathcal{V}_{\rho}({\mathcal T}_{\ast}f)(x):= \sup_{\epsilon_i\searrow0}
 \Big(\sum_{i=1}^{\infty}|T_{\epsilon_{i+1}}f(x)-T_{\epsilon_{i}}f(x)|^{\rho}\Big)^{1/\rho},
 \end{equation}
 where the supremum is taken over all sequences $ \{\epsilon_{i}\}$ decreasing to zero. The oscillation operator $\mathcal{O}({\mathcal T}_{\ast}f)$ can be introduced as
\begin{equation}\label{d-oci}
\mathcal{O}({\mathcal T}_{\ast}f)(x):=\Big(\sum_{i=1}^{\infty}\sup_{\epsilon_{i+1}\leq t_{i+1}
  <t_{i}\leq \epsilon_{i}}|T_{t_{i+1}}f(x)-T_{t_{i}}f(x)|^2\Big)^{1/2}
\end{equation}
 with $\{\epsilon_{i}\}$ being a fixed sequence decreasing to zero.

The $L^p$-boundedness of these operators were studied by Bourgain \cite{Bou} for
 $p=2$ and by Jones et al. \cite{jkrw} for $p\in[1, \fz)$ in the context of ergodic theory.
 Since then, in harmonic analysis, the study of boundedness of oscillation and variation operators associated with semigroups
 of operators and families of truncations of singular integrals have been paid more and more attention.
 In particular,  Campbell et al. \cite{cjrw1} first established the
strong $(p,\,p)$-boundedness in the range $1<p<\fz$ and the weak
type $(1,\,1)$-boundedness of the oscillation operator
 and the $\rho$-variation operator for the Hilbert transform. Subsequently, in \cite{cjrw2}, Campbell et al. further extended the results in \cite{cjrw1}
to the higher dimensional cases including  Riesz transforms and  general singular integrals with rough
homogeneous kernels in $\mathbb{R}^d$.
 On the other hand, Jones and Reinhold \cite{jr} obtained the $L^p$-boundedness properties for $p\geq1$
 of the oscillation and variation operators associated with the symmetric diffusion semigroup(see Lemma \ref{l-seimgroup bdd}).
  Crescimbeni et al. \cite{cmtv} subsequently established weighted variation inequalities of  heat
  semigroup and  Poisson semigroup associated to Laplacian and Hermite operator. For more results on oscillation and variation operators,
 we refer the readers to \cite{bct,b14,cmtt,jsw,LW,ZW} and the references therein.

 Let $\lz$ be a positive constant
%\in\rrp:=(0, \fz)$
and $\tbz$ be the Bessel operator
which is defined by setting, for all suitable functions $f$ on $\rrp :=(0, \fz)$,
%and $x\in \rrp$,
\begin{equation*}%\label{bessel 1}
\tbz f(x):=-\frac{d^2}{dx^2}f(x)-\frac{2\lz}{x}\frac{d}{dx}f(x).
\end{equation*}
An early work concerning the Bessel operator goes back to Muckenhoupt and Stein \cite{ms}.
They  developed a theory associated to
$\tbz$ which is parallel to the classical one associated to the Laplace operator $\triangle$.
 Since then, a lot of work concerning the Bessel operators was carried out; see, for example \cite{ak,bcfr,bfbmt,bfs,bhnv, dlwy, k78,v08,yy}.
In particular, Betancor et al. in \cite{bdt} established the
characterizations of the atomic Hardy space $H^1((0, \fz), \dmz)$
associated to $\tbz$ in terms of the Riesz transform and the radial
maximal function related to a class of functions including the Poisson semigroup $\{P_{t}^{[\lz]}\}_{t>0}$ and the heat
semigroup $\{W^{[\lz]}_t\}_{t>0}$ as special cases, where $\dmz(x):= x^{2\lz}\,dx$ and $dx$ is the Lebesgue measure.

The aim of this paper is to prove the $\lpz$-boundedness and their endpoint estimates of the oscillation and variation operators for
$\{P_{t}^{[\lz]}\}_{t>0}$ and $\{W^{[\lz]}_t\}_{t>0}$, respectively.  To this end, we recall some necessary notation.

Let $P_{\ast}^{[\lz]}:=\{\plz\}_{t>0}$ be a family of Poisson semigroup operators defined by
\begin{equation*}\label{DPoisson}
P_{t}^{[\lz]}f(x) := e^{-t\sqrt{\tbz}}f(x)=\int_0^{\infty}P_{t}^{[\lz]}(x,y)f(y)\,y^{2\lz}dy \ ,
\end{equation*}
where $J_\nu$ is the Bessel function of the first kind of order $\nu$ with $\nu\in(-1/2, \fz)$ and
\begin{equation}\label{kernel}
\begin{array}[b]{cl}
\plz(x, y)&=\dint_0^\fz e^{-tz}(xz)^{-\lz+1/2}J_{\lz-1/2}(xz)(yz)^{-\lz+1/2}
J_{\lz-1/2}(yz)\, dm_\lz(z)\\
&=\dfrac{2\lz t}{\pi}\dint_0^\pi\dfrac{(\sin\theta)^{2\lz-1}}
{(x^2+y^2+t^2-2xy\cos\theta)^{\lz+1}}\,d\theta,\,\,t,x,y\in(0,\,\infty);
\end{array}
\end{equation}
see \cite{bdt}.

Let $\{t_j\}_{j>0}$ be a fixed
decreasing sequence converging to zero and $\rho>2$. The $\rho$-variation operator $\varpz$ and
oscillation operator $\oscpz$ associated with the Poisson semigroup are defined by setting, for all suitable functions $f$ and $x\in \mathbb{R}_{+}$,
\begin{equation*}\label{varia}
\mathcal{V}_{\rho}\lf(P_{\ast}^{[\lz]}\r)f(x):=\sup_{t_j\searrow0}\lf(\sum_{j=1}^\fz\lf|P^{[\lz]}_{t_{j+1}}f(x)-P^{[\lz]}_{t_j}f(x)\r|^\rho\r)^{1/\rho}
\end{equation*}
and
\begin{equation*}\label{oscia}
\mathcal{O}\lf(P_{\ast}^{[\lz]}\r)f(x):=\lf(\sum_{j=1}^\fz\sup_{t_{j+1}\le \ez_{j+1}<\ez_j\le t_j}\lf|P^{[\lz]}_{\ez_{j+1}}f(x)-P^{[\lz]}_{\ez_{j}}f(x)\r|^2\r)^{1/2}.
\end{equation*}

The first main result of this paper is as follows.

\begin{thm}\label{t-bdd semigroup}
Let $\rho\in(2, \fz)$. The operators $\oscpz$ and $\varpz$ are both bounded
\begin{itemize}
  \item [{\rm (i)}] from  $\lpz$ to itself for any $p\in(1, \fz)$.
  \item [{\rm (ii)}] from $\loz$ to $L^{1,\,\fz}(\rrp,\,dm_\lz)$.
\end{itemize}
\end{thm}

For the endpoint $p=1$, we also consider the boundedness of $\oscpz$ and $\varpz$ on Hardy space $\hoz$ in \cite{bdt}.
Throughout this paper, for any $x$, $r\in \mathbb{R}_+$, we define $I(x, r):=(x-r, x+r)\cap \mathbb{R}_+$.

%
%
%The following definition of $H^1(\rrp,\,\dmz)$ was introduced by Bettancor et al. in \cite{bdt}.

%Characterizations of function spaces associated to the Bessel operator $\tbz$ were also studied
%by many authors. For example, Betancor et al. in \cite{bdt}
%characterized the atomic Hardy space $H^1(\mathbb{R}_+, \dmz)$
%associated to $\tbz$ in terms of the Riesz transform and the radial
%maximal function associated with the Hankel convolution of a class
%of suitable functions. The recent work \cite{dlwy} established a factorisation of the Hardy space associated to $\tbz$
%and a characterisation of the BMO space associated to $\tbz$ through commutators.

\begin{defn}[\cite{bdt}]\label{d-H^1}\rm
A measurable function $a$ is called an $H^1(\rrp,\,\dmz)$-atom if there exist a bounded interval $I\subset[0,\infty)$
such that
$$\supp a\subset I,\  \|a\|_{L^{\infty}(\rrp,\,\dmz)}\le 1/m_{\lz}(I) \,{\rm and}\, \int_{0}^{\infty}a(x)\dmz(x)=0.$$
 A function $f\in L^1(\rrp,\,\dmz)$ is in $H^1(\rrp,\,\dmz)$ if and only if $f(x)=\sum_{j=1}^{\infty}\alpha_ja_j(x)$ in $L^1(\rrp,\,\dmz)$,
 where for every $j$, $a_j$ is an $H^1(\rrp,\dmz)$-atom and $\alpha_j\in\mathbb{C}$, with $\sum_{j=1}^{\infty}|\alpha_j|<\infty$.
 The norm $\|f\|_{H^1(\rrp,\,\dmz)}$ is defined by
\begin{equation*}\label{dH^1}
\|f\|_{H^1(\rrp,\,\dmz)}:=\inf\sum_{j=1}^{\infty} |\alpha_j| ,
\end{equation*}
where the infimum is taken over all possible decompositions of $f$ as above.
 \end{defn}

%The aim of this paper is to build up the $L^p(\mathbb R_+, \dmz)$-boundedness
%properties and the $BMO(\rrp,\dmz)$-boundedness for the oscillation operator $\oscpz$ and $\rho$-variation operator $\varpz$.
%Base on the $L^p(\mathbb R_+, \dmz)$-boundedness properties, the $(H^1(\rrp,\dmz),\,L^1(\rrp,\dmz))$-type endpoint
%estimation  also be established. Our main results can be formulated as follows.

We now state the boundedness of $\oscpz$ and $\varpz$ from $H^1(\rrp,\,\dmz)$ introduced in \cite{bdt} to $\loz$.

\begin{thm}\label{t-H1 semigroup}
Let $\rho\in(2, \fz)$. The operators $\oscpz$ and $\varpz$ are both bounded from $\hoz$ to $\loz$.
\end{thm}

As a consequence of Theorem \ref{t-H1 semigroup}, we have the following characterization of
$\hoz$ via $\mathcal V_\rho\lf(P_{\ast}^{\lz}\r)$.

\begin{thm}\label{t-char semigroup}
Let $\rho\in(2,\fz)$. A function  $f\in\loz$ is in $\hoz$ if and only if
$\mathcal V_\rho(P_{\ast}^{[\lz]})(f)\in \loz$. Moreover, there exists a positive constant $C>1$ such that
$$\|f\|_\hoz/C\le \|f\|_\loz+\lf\|\mathcal V_\rho\lf(P_{\ast}^{[\lz]}\r)(f)\r\|_\loz\le C\|f\|_\hoz.$$
\end{thm}

For $p=\fz$, we study the boundedness of $\oscpz$ and $\varpz$  on $\bmoz$ in \cite{bdt}.

\begin{defn}[\cite{yy}]\label{d-bmo}\rm
A function $f\in L^1_{\rm loc}(\rrp,\,dm_\lambda)$ belongs to
the {space} $\bmoz$ if
\begin{equation*}\label{mofi}
\|f\|_{\rm BMO(\rrp,\,\dmz)}:=\sup_{x,\,r\in (0,\,\infty)}
\frac 1{m_\lambda(I(x,r))}\int_{I(x,\,r)}|f(y)-f_{I(x,\,r),\,\lambda}|\,y^{2\lambda}dy<\infty,
\end{equation*}
where
\begin{equation}\label{f-i}
f_{I(x,\,r),\,\lambda}:=\frac 1{m_{\lambda}(I(x,r))}\int_{I(x,\,r)}f(y)\,y^{2\lambda}dy.
\end{equation}
\end{defn}

Our result concerning the boundedness of $\oscpz$ and $\varpz$ on $\bmoz$ is stated as below.

\begin{thm}\label{t-bmo semigroup}
Let $\rho\in(2, \fz)$. The operators $\oscpz$ and $\varpz$ are both bounded  on $\bmoz$.
\end{thm}

\begin{rem}\rm\label{r-heat semig}
Let $\{W^{[\lz]}_t\}_{t>0}$ be the heat semigroup associated with $\Delta_\lz$ defined by
setting, for all $f\in \bigcup_{1\le p\le\fz}\lpz$ and $x\in\rrp$,
$$W^{[\lz]}_t f(x):=e^{-t\Delta_\lz}f(x)=\int_0^\fz W^{[\lz]}_t(x, y) f(y)\,\dmz(y),$$
where
$$W^{[\lz]}_t(x, y):=\frac{2^{(1-2\lz)/2}}{\Gamma(\lz)\sqrt \pi}t^{-\lz-\frac12}\exp\lf(-\frac{x^2+y^2-2xy\cos\theta}{2t}\r)
(\sin\theta)^{2\lz-1}\,d\theta;$$
(see \cite[pp.\,200-201]{bdt}).
We remark all the conclusions of Theorems \ref{t-bdd semigroup}, \ref{t-H1 semigroup}, \ref{t-char semigroup} and \ref{t-bmo semigroup} hold
if the Poisson semigroup $\{\plz\}_{t>0}$ is replaced by the heat semigroup $\{W^{[\lz]}_t\}_{t>0}$.
\end{rem}

The organization of this paper is as follows.

 Section \ref{s3} is devoted to the proof of  Theorem \ref{t-bdd semigroup}. For this purpose,
 we first establish a basic proposition on the upper bounds of the Poisson kernel and its derivatives.
 Then we prove the $L^p(\mathbb R_+,\, \dmz)$ boundedness $(p>1)$ of $\oscpz$ and $\varpz$ by
 using the result in \cite{jr} on the oscillation and variation for symmetric diffusion semigroups
(see Lemma \ref{l-seimgroup bdd}).  Moreover, using the Calder\'on-Zygmund decomposition, the basic
proposition and the $L^p(\mathbb R_+,\, \dmz)$ boundedness $(p>1)$ of $\oscpz$ and $\varpz$,
we further show that both $\oscpz$ and $\varpz$ are of weak type (1,1).

 In Section \ref{s4},  we establish another endpoint estimation for $p=1$  of $\oscpz$ and $\varpz$. Precisely,  applying
 the atomic  decomposition and the basic proposition in Section \ref{s3},
  we prove $\oscpz$ and $\varpz$ are bounded from $H^1(\rrp,\dmz)$
 to $\loz$. As an application, by the characterization of
$\hoz$ space via the maximal operator $\cm_{P^{[\lz]}}$ of the Poisson semigroup $\{P_{t}^{[\lz]}\}_{t>0}$ in \cite{bdt},
we further establish an equivalent characterization of
$\hoz$ space in terms of ${\mathcal V}_\rho(P^{[\lz]}_\ast)$.

Based on the proposition in Section \ref{s3} and properties of ${\rm BMO}(\rrp,\,\dmz)$, we in Section \ref{s5} obtain the boundedness
of $\oscpz$ and $\varpz$  on ${\rm BMO}(\rrp,\,\dmz)$. We borrow the ideas in \cite{b14}. However, compared
with the case in \cite{b14}, since the Poisson semigroup $\{P^{[\lz]}_t\}_{t>0}$ associated with $\Delta_\lz$
has the conservation property $P^{[\lz]}_t(1)=1$, our argument is more straightfoward than that in \cite{b14}.

Throughout the paper,
we denote by $C$ {positive constants} which
is independent of the main parameters, but it may vary from line to
line. For every $p\in(1, \fz)$, we denote by $p'$ the conjugate of $p$, i.e., $1/p'+1/p=1$.
If $f\le Cg$, we then write $f\ls g$ or $g\gs f$;
and if $f \ls g\ls f$, we  write $f\sim g.$ For any $k\in \mathbb{R}_+$ and $I:= I(x, r)$ for some $x$, $r\in (0, \fz)$,
$kI:=I(x, kr)$. For any $x,\,r\in(0,\,\infty)$, if  $x<r$, then
$$I(x,\,r)=(0,\,x+r)=I\lf(\frac{x+r}{2},\,\frac{x+r}{2}\r).$$
Thus, we may assume that $x\ge r$.

%%%%%%%%%%%%%%%%%%%%%%%%%%%%%%%%%%%%%%%%%%%%%%%%%%%%%%%%%%%%%%%%%%%%%%%%%%%%%%%%%%%%%%%%%%%%%%%%%%%%%%%%%%%%%%%%%%%
%%%%%%%%%%%%%%%%%%%%%%%%%%%%%%%%%%%%%%%%%%%%%%%%%%%%%%%%%%%%%%%%%%%%%%%%%%%%%%%%%%%%%%%%%%%%%%%%%%%%%%%%%%%%%%%%%%%%%
\section{$L^p(\rrp,\dmz)$-boundedness and weak type (1,1) estimates}\label{s3}

In this section, we provide the proof of Theorem \ref{t-bdd semigroup}. To begin with, we first establish
a basic proposition on the upper bounds of the Poisson kernel and its derivatives,
which is a useful tool in this paper.

\begin{prop}\label{p-upp bdd poisson}
There exists a positive constant $C$ such that for any $x,\,y,\,t\in(0, \fz)$,
\begin{enumerate}
  \item [\rm i)]
  \begin{equation}\label{p-t1}
\lf|P_{t}^{[\lambda]}(x,y)\r|\le C\dfrac{t}{(|x-y|^2+t^2)^{\lambda+1}},
\end{equation}
and
\begin{equation}\label{p-t2}
\lf|P_{t}^{[\lambda]}(x,y)\r|\le C\dfrac{t}{(xy)^{\lambda}(|x-y|^2+t^2)}.
\end{equation}
  \item [\rm ii)] $$\lf|\partial_x\plz(x,y)\r|\le C\dfrac{t}{(|x-y|^2+t^2)^{\lz+3/2}},$$ and
$$\lf|\partial_x\plz(x,y)\r|\le C\dfrac{t}{(x y)^{\lz}(|x-y|^2+t^2)^{3/2}}.$$
  \item [\rm iii)]$$\lf|\partial_t\plz(x,y)\r|\le C\dfrac1{(|x-y|^2+t^2)^{\lz+1}},$$ and
$$\lf|\partial_t\plz(x,y)\r|\le C\dfrac1{(x y)^{\lz}(|x-y|^2+t^2)}.$$

  \item[\rm iv)]
  \begin{equation}\label{p-p-t1}
  \lf|\partial_y\partial_{t}\plz(x,y)\r|+\lf|\partial_x\partial_{t}\plz(x,y)\r|\le C\dfrac1{(|x-y|^2+t^2)^{\lz+3/2}},
  \end{equation}
   and
\begin{equation}\label{p-p-t2}
\lf|\partial_y\partial_{t}\plz(x,y)\r|+\lf|\partial_x\partial_{t}\plz(x,y)\r|\le C\dfrac1{(x y)^{\lz}(|x-y|^2+t^2)^{3/2}}.
\end{equation}
\end{enumerate}
\end{prop}

\begin{proof}
We first show {\rm i)}.
By \eqref{kernel} and the fact
\begin{equation}\label{sin-t}
\dint_{0}^{\pi}(\sin\theta)^{2\lz-1}d\theta=\dfrac{\Gamma(\lz)\sqrt{\pi}}{\Gamma(\lz+1/2)},
\end{equation}
it is easy to see  \eqref{p-t1} holds.  On the other hand,
by the fact that for $\theta\in [0,\,\pi/2]$, $\sin\theta\sim\theta$ and $1-\cos \theta \geq 2(\theta/\pi)^2$, we have
\begin{eqnarray*}
\lf|P_{t}^{[\lambda]}(x,y)\r|&&\ls\dint_{0}^{\pi/2}\dfrac{t(\sin\theta)^{2\lz-1}}{[|x-y|^2+t^2+2xy(1-\cos\theta)]^{\lz+1}}d\theta\\
&&\ls\dint_{0}^{\pi/2}\dfrac{t\theta^{2\lz-1}}{[|x-y|^2+t^2+4xy\theta^2/{\pi^{2}}]^{\lz+1}}d\theta\\
&&\lesssim \dfrac{t}{(xy)^{\lz}(|x-y|^2+ t^2)}\dint_{0}^{\infty}\frac{\beta^{2\lz-1}}{(1+\beta^2)^{\lz+1}}\,d\beta\\
 &&\lesssim\dfrac{t}{(x y)^{\lz}(|x-y|^2+t^2)}.
\end{eqnarray*}
This implies \eqref{p-t2} and hence {\rm i)}.

Observe that
\begin{eqnarray*}
\lf|\partial_t\plz(x,y)\r|\ls\dint_{0}^{\pi}\dfrac{(\sin\theta)^{2\lz-1}}{(x^2+y^2+t^2-2xy\cos\theta)^{\lz+1}}\,d\theta,
\end{eqnarray*}
 \begin{equation*}
\lf|\partial_x\plz(x,y)\r|\lesssim \dint_{0}^{\pi}\dfrac{t(\sin\theta)^{2\lz-1}}{(x^2+y^2+t^2-2xy\cos\theta)^{\lz+3/2}}\,d\theta,
 \end{equation*}
 and
\begin{equation*}
\lf|\partial_x\partial_{t}\plz(x,y)\r|+\lf|\partial_y\partial_{t}\plz(x,y)\r|\lesssim
\dint_{0}^{\pi}\dfrac{(\sin\theta)^{2\lz-1}}{(x^2+y^2+t^2-2xy\cos\theta)^{\lz+3/2}}\,d\theta.
\end{equation*}
Then using a similar argument, we see that {\rm ii)}-{\rm iv)} hold. This finishes the proof of Proposition \ref{p-upp bdd poisson}.
\end{proof}

Before we present the proof of Theorem \ref{t-bdd semigroup}, we also need two auxiliary lemmas
which were established in \cite{jr} and \cite{cw71}.

\begin{lem}\label{l-seimgroup bdd}
Let $(\Sigma, d\mu)$ be a positive measure space and $T_\ast=\{T_t\}_{t>0}$
a symmetric diffusion semigroup satisfying that $T_tT_s=T_sT_t$ for any $t,\,s\in(0, \fz)$,
$T_0=\mathfrak{I}$ the identity operator, $\lim_{t\to0}T_tf=f$ in $L^2(\Sigma,\,d\mu)$ and
\begin{itemize}
  \item [${\rm (T_i)}$]$\|T_tf\|_{L^p(\Sigma,\,d\mu)}\le \|f\|_{L^p(\Sigma,\,d\mu)}$ for all $p\in[1, \fz]$ and $t\in(0,\,\fz)$;
  \item [${\rm(T_{ii})}$]$T_t$ is self-adjoint on $L^2(\Sigma,\,d\mu)$ for all $t\in(0,\, \fz)$;
  \item[${\rm (T_{iii})}$] $T_t f\ge0$ for all $f\ge0$ and $t\in(0, \fz)$;
  \item [${\rm (T_{iv})}$] $T_t(1)=1$ for all $t\in (0,\,\fz)$.
\end{itemize}
Then the operators ${\mathcal V}_\rho(T_\ast)$ and ${\mathcal O}(T_\ast)$ defined
in \eqref{d-var} and \eqref{d-oci} are both bounded on $L^p(\Sigma,\,d\mu)$
for any $p\in(1, \fz)$.
\end{lem}

%\begin{lem}\label{Whiteney-lemma}(\rm{\bf{Whitney type covering Lemma}})
%Suppose $(X,\mu, d)$ is a space of homogenous type, where $X$ a topological space, $\mu$ a Borel measure and $d$ a quasi-metric(or quasi-distance). $U\subsetneq X$ is an open bounded set and $C\geq 1$. Then there exists a sequence of spheres $\{S_j\}=\{S_{s_j}(x_j)\}$ satisfying
%\begin{itemize}
%  \item [${\rm (i)}$] $U=\bigcup_jS_j$;
%  \item [${\rm({ii})}$] there exists a constant $M$ such that no point of $X$ belongs to more than $M$ of the spheres $\bar{S}_j\equiv S_{Cs_j}(x_j)$;
%  \item[${\rm ({iii})}$] $\bar{\bar{S}}_j\cap(X-U)\neq\emptyset$ for each $j$, where $\bar{\bar{S}}_j=S_{3KCs_j}(x_j)$
%and $K$ is a constant occurring in the "triangle inequality" satisfied by the quasi-distance $d$.
%  \end{itemize}
%\end{lem}
%Property ${\rm ({ii})}$ be referred to as the $M$-disjointness of the collection $\{\bar{S}_j\}$.
It is straightforward from the definition of  $m_\lambda$ (i.e., $dm_\lambda(x):=x^{2\lambda}dx$)  that
 there exists a finite constant $C > 1$ such that for all $x,\,r\in\mathbb{R}_+$,
\begin{equation*}\label{volume property-1}
 C^{-1}m_\lz(I(x, r))\le x^{2\lz}r+r^{2\lz+1}\le C m_\lz(I(x, r)).
\end{equation*}
%Moreover, it was shown in \cite{dlmwy} that for any interval $I\subset\mathbb{R}_{+}$,
%\begin{equation*}\label{reverse doubl}
%\min\lf(2, 2^{2\lz}\r)m_\lz(I)\le m_\lz(2I)\le 2^{2\lz+1}m_\lz(I),
%\end{equation*}
This means that $(\mathbb{R}_+, |\cdot|, dm_\lz)$ is a space of homogeneous type
in the sense of \cite{cw71,cw77}.

The following Calder\'on-Zygmund decomposition was established in \cite[pp.\,73-74]{cw71} in the setting of spaces of homogeneous type.

\begin{lem}\label{l-cz} Let  $f\in\loz$ and
$\eta>0$, there exist a family
of intervals $\{I_j\}_{j}$, and constants $C>0$ and $M\ge 1$, such
that
\begin{itemize}
\item [\rm (i)] $f=:g+b=:g+\sum_jb_j$, where $b_j$ is supported in $I_j$,
\item [\rm (ii)] $\|g\|_\linz\le C\eta$ and $\|g\|_\loz\le C\|f\|_\loz$,
\item [\rm (iii)] $\sum_j\|b_j\|_\loz\le C\|f\|_\loz$ and $\int_{I_j} b_j(x)\,\dmz(x)=0$
for each $j$,
\item [\rm (iv)] $\sum_jm_\lz(I_j)\le \frac C\eta\|f\|_\loz$,
\item [\rm (v)] for any $x\in\rrp$, $\sum_j\chi_{I_j}(x)\le M$.
\end{itemize}
\end{lem}

\begin{proof}[Proof of Theorem \ref{t-bdd semigroup}]
We first claim that  $\{\plz\}_{t>0}$ satisfies the conditions in Lemma \ref{l-seimgroup bdd}.
In fact,  $\{P^{[\lz]}_t\}_{t>0}$ is a contraction semigroup on $\lrz$
for all $r\in[1, \fz]$; see \cite[p.\,197]{bdt}. Thus ${\rm(T_i)}$ holds.
Next, Since $\Delta_\lz$ is a self-adjoint  operator, then
 $\{P^{[\lz]}_t\}_{t>0}$ satisfies ${\rm(T_{ii})}$.
 Third, recall that for any $t,\,x,\,y\in(0, \fz)$,
the kernel $P^{[\lz]}_t(x, y)$ of $P^{[\lz]}_t$ satisfies \eqref{kernel}, which implies
${\rm(T_{iii})}$.  Moreover, by \cite[p.\,208]{bdt}, the conservation property ${\rm(T_{iv})}$ holds
for $\{P^{[\lz]}_t\}_{t>0}$.
According to Lemma \ref{l-seimgroup bdd}, we conclude that $\oscpz$
and $\varpz$ are bounded on $L^p(\rrp,\dmz)$ for any $p\in(1,\infty)$.

 We next establish the weak (1,1) estimation.
We only give the proof of $\oscpz$ and the proof of $\varpz$ is similar.
For any $f\in\loz$ and $\eta\in(0, \fz)$. By Lemma \ref{l-cz}, we have functions $g, b$, $b_j$
and intervals $\{I_j\}_j$ as in Lemma \ref{l-cz}.

Since the operator $\oscpz$  is subadditive, to show
$$m_{\lz}\lf(\lf\{x\in\rrp: {\mathcal O\lf(P^{[\lz]}_\ast\r)} f(x)>\eta\r\}\r)\leq \frac{C}{\eta}\|f\|_{\loz},$$
it suffices to prove
 \begin{equation}{\label{rg}}
m_{\lz}\Big(\Big\{x\in\rrp: \lf|\boscpz(g)(x)\r|>\frac{\eta}{2}\Big\}\Big)\leq \frac{C}{\eta}\|f\|_{\loz},
\end{equation}
 and
\begin{equation}\label{rb}
m_{\lz}\Big(\Big\{x\in\rrp: \lf|\boscpz(b)(x)\r|>\frac{\eta}{2}\Big\}\Big)\leq \frac{C}{\eta}\|f\|_{\loz}.
\end{equation}
For \eqref{rg}, by the $L^{2}$-boundedness of $\oscpz$ and Lemma \ref{l-cz} {\rm (ii)}, we have
\begin{eqnarray*}
m_{\lz}\Big(\Big\{x\in\rrp: \lf|\boscpz(g)(x)\r|>\dfrac{\eta}{2}\Big\}\Big)
&&\lesssim \frac{1}{{\eta}^2}\dint_{\rrp}\lf|g(x)\r|^{2}\,\dmz(x)\noz\\
&&\lesssim\frac{1}{{\eta}^2}\lf\|g\r\|_{L^{\infty}(\rrp,\,\dmz)}\dint_{\rrp}\lf|g(x)\r|\,\dmz(x)\noz\\
&&\lesssim \frac{1}{{\eta}} \dint_{\rrp}|f(x)|\,\dmz(x).
\end{eqnarray*}
This shows \eqref{rg}.

Now, we prove \eqref{rb}. Let $\tilde{I_{j}}:= 3I_j$ and $\mathcal{\widetilde{I}}:=\bigcup_{j}\tilde{I_{j}}$.
Using the doubling property of $m_{\lz}$ and Lemma \ref{l-cz} {\rm (iv)}, we write
\begin{eqnarray*}
&&m_{\lz}\lf(\lf\{x\in\rrp:\,\lf|\boscpz(b)(x)\r|>\dfrac{\eta}{2}\r\}\r)\\
&&\quad\lesssim m_{\lz}\lf(\widetilde{\mathcal{I}}\r)+m_{\lz}
\lf(\lf\{x\in\rr_{+}\backslash\widetilde{\mathcal{I}}:\,\lf|\boscpz(b)(x)\r|>\dfrac{\eta}{2}\r\}\r)\\
&&\quad\lesssim\frac1\eta\|f\|_\loz+m_\lz
\lf(\lf\{x\in\rr_{+}\backslash\mathcal{\widetilde{I}}:\lf|\boscpz(b)(x)\r|>\frac{\eta}{2}\r\}\r).
\end{eqnarray*}
It remains to estimate the last term. For each $j$, by the fundamental theorem of calculus, we now analyze the operator
\begin{eqnarray*}
\boscpz(b_j)(x)&&=\sum_{k=1}^{\infty}\lf(\sup_{t_{k+1}\le\varepsilon_{k+1}<\varepsilon_{k}
\le t_{k}}\lf|P_{\varepsilon_{k+1}}^{[\lambda]}b_j(x)-P_{\varepsilon_k}^{[\lambda]}b_j(x)\r|^2\r)^{1/2}\\
&&\le\sum_{k=1}^{\infty}\lf[\sup_{t_{k+1}\le\varepsilon_{k+1}<\varepsilon_{k}\le t_{k}}
\lf|\int_{\rrp}\lf(P_{\varepsilon_{k+1}}^{[\lambda]}(x,y)-P_{\varepsilon_k}^{[\lambda]}(x,y)\r)b_j(y)\,\dmz(y)\r|\r]\\
&&=\sum_{k=1}^{\infty}\lf[\sup_{t_{k+1}\le\varepsilon_{k+1}<\varepsilon_{k}\le t_{k}}
\lf|\int_{\rrp}\int_{\varepsilon_{k+1}}^{\varepsilon_{k}}\partial_{t}P_{t}^{[\lambda]}(x,y)b_j(y)\,dt\dmz(y)\r|\r].
\end{eqnarray*}
By (i) and (iii) of Lemma \ref{l-cz}, we have
\begin{eqnarray*}
\int_{\rrp}\partial_{t}P_{t}^{[\lambda]}(x,y)b_j(y)\dmz(y)=\int_{\rrp}\lf[\partial_tP_{t}^{[\lambda]}(x,y)-\partial_tP_{t}^{[\lambda]}(x,y_j)\r]b_j(y)\,\dmz(y),
\end{eqnarray*}
where $y_j$ is the center of $I_j$.
This yields that
\begin{eqnarray*}
\boscpz(b_j)(x)\le\int_{0}^{\infty}\int_{\rrp}\lf|\partial_tP_{t}^{[\lambda]}(x,y)-\partial_tP_{t}^{[\lambda]}(x,y_j)\r||b_j(y)|\,\dmz(y)\,dt.
\end{eqnarray*}
Then applying Lemma \ref{l-cz} (i) and the mean value theorem, we write
\begin{eqnarray*}
&&m_{\lz}\lf(\lf\{x\in\rr_{+}\backslash\,\mathcal{\widetilde{I}}:\lf|\boscpz(b)(x)\r|>\frac{\eta}{2}\r\}\r)\\
&&\quad\lesssim \frac{1}{\eta}\int_{\rrp\backslash\,\mathcal{\widetilde{I}}}\lf|\boscpz(b)(x)\r|\,\dmz(x)\\
&&\quad\lesssim\frac{1}{\eta}\sum_j\int_{\rrp\backslash\,\tilde{I_j}}\int_0^{\infty}
\int_{\rrp}\lf|\partial_tP_{t}^{[\lambda]}(x,y)-\partial_tP_{t}^{[\lambda]}(x,y_j)\r|\lf|b_j(y)\r|\,\dmz(y)\,dt\,\dmz(x)\\
&&\quad\lesssim\frac{1}{\eta}\sum_j\int_{\rrp\backslash\,\tilde{I_j}}\int_0^{\infty}
\int_{\rrp}\lf|\partial_y\partial_tP_{t}^{[\lambda]}(x,\xi)\r|\lf|y-y_j\r|\lf|b_j(y)\r|\,\dmz(y)\,dt\,\dmz(x)\\
&&\quad\lesssim\frac{1}{\eta}\sum_j\int_{\rrp\backslash\,\tilde{I_j}}\int_0^{|I_j|}
\int_{\rrp}\lf|\partial_y\partial_tP_{t}^{[\lambda]}(x,\xi)\r|\lf|y-y_j\r|\lf|b_j(y)\r|\,\dmz(y)\,dt\,\dmz(x)\\
&&\quad\qquad+\frac{1}{\eta}\sum_j\int_{\rrp\backslash\,\tilde{I_j}}
\int_{|I_j|}^{\infty}\int_{\rrp}\lf|\partial_y\partial_tP_{t}^{[\lambda]}(x,\xi)\r|\lf|y-y_j\r|\lf|b_j(y)\r|\,\dmz(y)\,dt\,\dmz(x)\\
&&\quad=:{\rm T}_1+{\rm T}_2,
\end{eqnarray*}
 where for $y\in I_j$, $\xi:=sy+(1-s)y_j$ for certain $s\in(0, 1)$.

To estimate ${\rm T}_1$, we first claim that
\begin{equation}\label{upp bdd deriva meas}
\lf|\partial_y\partial_{t}\plz(x,\xi)\r|\ls \frac1{m_\lz(I(y_j, |x-y_j|))}\frac1{(|x-y_j|+t)^2}.
\end{equation}
In fact, for any $x\in\rrp\setminus \tilde I_j$ and $y,\xi\in I_j$, we have
$$|x-y|\sim|x-y_j|\sim |x-\xi|$$
and
\begin{align}\label{equiv meas}
m_{\lz}(I(y_j,|x-y_j|))\sim m_{\lz}(I(x,|x-y_j|))\sim m_{\lz}(I(x,|x-y|)).
\end{align}
We further consider the following two cases:

Case 1: $x\le 2|x-y|$. In this case,
\begin{equation*}
m_{\lz}(I(x,|x-y|))\sim |x-y|^{2\lz+1}.
\end{equation*}
By Proposition \ref{p-upp bdd poisson} {\rm iv)} and \eqref{equiv meas}, we see that
\begin{eqnarray*}
\lf|\partial_y\partial_{t}\plz(x,\xi)\r|&\ls&\frac1{(|x-\xi|^2+t^2)^{\lz+3/2}}\\
&\lesssim&\frac1{m_\lz(I(x, |x-y|))}\frac1{(|x-y|+t)^2}\\
&\sim&\frac1{m_\lz(I(y_j, |x-y_j|))}\frac1{(|x-y_j|+t)^2}.
\end{eqnarray*}

Case 2: $x\geq 2|x-y|$. In this case, $m_{\lz}(I(x,|x-y|))\sim x^{2\lz}|x-y|$.
Since $x\notin\tilde{I}_j$ and $y\in I_j$, $x\sim y\sim\xi$, we have that by \eqref{equiv meas},
$$m_\lz(I(y_j,|x-y_j|))\sim m_\lz(I(x,|x-y|))\sim (x\xi)^{\lz}|x-y|.$$
Thus, applying \eqref{p-p-t2}, we conclude that
\begin{eqnarray*}
\lf|\partial_y\partial_{t}\plz(x,\xi)\r|&\ls&\frac1{(x\xi)^{\lambda}(|x-\xi|+t)^{3}}\\
&\ls&\frac1{m_\lz(I(x, |x-y|))}\frac1{(|x-y|+t)^2}\\
&\sim&\frac1{m_\lz(I(y_j, |x-y_j|))}\frac1{(|x-y_j|+t)^2}.
\end{eqnarray*}
Combining the two cases above, we conclude that \eqref{upp bdd deriva meas} holds.

For $x\notin \tilde{I_j}$ and $y\in I_j$,
by \eqref{upp bdd deriva meas} and Lemma \ref{l-cz} (iii), we have  that
\begin{eqnarray*}
{\rm T}_1&&\lesssim \frac{1}{\eta}\sum_j\sum_{k=1}^{\infty}\int_{3^{k+1}I_j\backslash\,3^kI_j}\int_{0}^{|I_j|}
\int_{\rrp}|I_j|\lf|\partial_y\partial_tP_{t}^{[\lambda]}(x,\xi)\r||b_j(y)|\,\dmz(y)\,dt\,\dmz(x)\\
%&&\lesssim\frac{1}{\eta}\sum_j\sum_{k=1}^{\infty}\int_{3^{k+1}I_j\backslash\,3^kI_j}\int_{\rrp}\frac{|I_j|^2}{|x-y|^{2\lz+3}}|b_j(y)|\,\dmz(y)\,\dmz(x)\\
&&\lesssim\frac{1}{\eta}\sum_j\sum_{k=1}^{\infty}\int_{3^{k+1}I_j\backslash\,3^kI_j}\int_{\rrp}\frac{|I_j|^2}{m_{\lz}(I(y_j,|x-y_j|))|x-y_j|^2}|b_j(y)|\,\dmz(y)\,\dmz(x)\\
&&\lesssim\frac{1}{\eta}\sum_j\int_{\rrp}|b_j(y)|\dmz(y)\sum_{k=1}^{\infty}\frac{|I_j|^2}{(3^k|I_j|)^2}
\dfrac{m_{\lz}(3^{k+1}I_j)}{m_{\lz}(3^kI_j)}\\
&&\lesssim \frac{1}{\eta}\sum_j\int_{\rrp}|b_j(y)|\,\dmz(y)\\
&&\lesssim \frac{1}{\eta}\|f\|_{\loz}.
\end{eqnarray*}

Now we estimate ${\rm T}_2$. By \eqref{upp bdd deriva meas}  and Lemma \ref{l-cz} (iii) again, we have
\begin{eqnarray*}
{\rm T}_2&&\lesssim \frac{1}{\eta}\sum_j\sum_{k=1}^{\infty}\int_{3^{k+1}I_j\backslash\,3^kI_j}
\int_{|I_j|}^{\infty}\int_{\rrp}\frac1{m_\lz(I(y_j, |x-y_j|))}\frac{|I_j|}{(|x-y_j|+t)^2}|b_j(y)|\,\dmz(y)\,dt\,\dmz(x)\\
&&\lesssim\frac{1}{\eta}\sum_j\sum_{k=1}^{\infty}\int_{3^{k+1}I_j\backslash\,3^kI_j}
\int_{\rrp}\frac1{m_\lz(I(y_j, |x-y_j|))}\frac{|I_j|}{(|x-y_j|+|I_j|)}|b_j(y)|\,\dmz(y)\,\dmz(x)\\
&&\lesssim\frac{1}{\eta}\sum_j\int_{\rrp}|b_j(y)|\,\dmz(y)\sum_{k=1}^{\infty}
\frac{|I_j|m_{\lz}(3^{k+1}I_j)}{m_{\lz}(3^kI_j)3^k|I_j|}\\
&&\lesssim\frac{1}{\eta}\|f\|_{\loz}.
\end{eqnarray*}
Consequently, we obtain ${\rm T}_2\lesssim \frac{1}{\eta}\|f\|_{\loz}$. Therefore, we have
\begin{eqnarray*}
m_{\lz}\Big(\Big\{x\in\rr_{+}\backslash \widetilde{\mathcal{I}}:\lf|{\mathcal O\lf(P^{[\lz]}_\ast\r)}(b)(x)\r|>\frac{\eta}{2}\Big\}\Big)
\lesssim \frac{1}{\eta}\|f\|_{\loz}.
\end{eqnarray*}
This finishes the proof of Theorem \ref{t-bdd semigroup}.
\end{proof}

%%%%%%%%%%%%%%%%%%%%%%%%%%%%%%%%%%%%%%%%%%%%%%%%%%%%%%%%%%%%%%%%%%%%%%%%%%%%%%%%%%%%%%%%%%%%%%%%%%%%%%%%%%%%%%
%%%%%%%%%%%%%%%%%%%%%%%%%%%%%%%%%%%%%%%%%%%%%%%%%%%%%%%%%%%%%%%%%%%%%%%%%%%%%%%%%%%%%%%%%%%%%%%%%%%%%%%%%%%%%%%

\section{$(\hoz, \loz)$-boundedness}\label{s4}

In this section, we present the proofs of Theorems \ref{t-H1 semigroup} and \ref{t-char semigroup}.
The proof for $\varpz$ is completely analogous to that of $\oscpz$ in Theorem \ref{t-H1 semigroup}.
Hence, we only provide the version for $\oscpz$ here.

\begin{proof}[Proof of Theorem \ref{t-H1 semigroup}]
Assume that $f\in \hoz$. Then we have that $f=\sum \alpha_ka_k$, where $a_k$ is
an $\hoz$-atom such that $\sum_k|\alpha_k|<\infty$ and there exists an open bounded interval $I_k=I(x_k, r)\subset\rrp$ such
that $\supp(a_k)\subset I$, $\|a_k\|_{L^{\infty}(\rrp,\,\dmz)}\le [m_{\lz}(I_k)]^{-1}$ and $\int_\rrp a_k(x)\dmz(x)=0$.
Since $\hoz\subset \loz$ and  $\plz$ is bounded on $\loz$ for each $t\in(0, \fz)$, we have $\plz f=\sum_{k}\alpha_k\plz (a_k)$.
Moreover, we write
\begin{eqnarray*}
\boscpz f(x)&&=\lf[\sum_{j=1}^{\infty}\lf(\sup_{t_{j+1}
\le\varepsilon_{j+1}<\varepsilon_{j}\le t_{j}}\lf|\sum_k\alpha_kP_{\varepsilon_j}^{[\lz]}
(a_k)(x)-\sum_k\alpha_kP_{\varepsilon_{j+1}}^{[\lz]}(a_k)(x)\r|\r)^2\r]^{1/2}\\
&&\le \lf[\sum_{j=1}^{\infty}\lf(\sum_k|\alpha_k|\sup_{t_{j+1}
\le\varepsilon_{j+1}<\varepsilon_{j}\le t_{j}}\lf|P_{\varepsilon_j}^{[\lz]}(a_k)(x)-P_{\varepsilon_{j+1}}^{[\lz]}(a_k)(x)\r|\r)^2\r]^{1/2}.
\end{eqnarray*}

Applying Minkowski's inequality for series, we have
\begin{eqnarray*}
\boscpz f(x)&&\le \sum_k|\alpha_k|\lf[\sum_{j=1}^{\infty}\sup_{t_{j+1}\le
\varepsilon_{j+1}<\varepsilon_{j}\le t_{j}}\lf|P_{\varepsilon_j}^{[\lz]}a_k(x)-P_{\varepsilon_{j+1}}^{[\lz]}a_k(x)\r|^2\r]^{1/2}\\
&&=\sum_k|\alpha_k|\boscpz (a_k)(x).
\end{eqnarray*}
Therefore,
\begin{eqnarray*}
&&\int_{\rrp}\boscpz f(x)\dmz(x)\\
&&\quad\le \sum_{k}|\alpha_k|\dint_{\rrp}\boscpz(a_k)(x)\,\dmz(x)\\
&&\quad\le\sum_{k}|\alpha_k|\lf[\dint_{2I_k}\boscpz (a_k)(x)\dmz(x)+\dint_{\rrp\backslash 2I_k}\boscpz(a_k)(x)\,\dmz(x)\r]\\
&&\quad=:\sum_k|\alpha_k|({\rm E}_{k,1}+{\rm E}_{k,2}).
\end{eqnarray*}
By H\"older's inequality, $\|a_k\|_{L^{\infty}(\rrp,\,\dmz)}\le [m_{\lz}(I_k)]^{-1}$ and the $\ltz$-boundedness
property of $\oscpz$, we obtain
\begin{eqnarray*}
{\rm E}_{k,1}&&\le \lf\{\int_{2I_k}\lf[\mathcal O\lf(P^{[\lz]}_\ast\r)(a_k)(x)\r]^2\,\dmz(x)\r\}^{1/2}\lf[m_{\lz}(2I_k)\r]^{1/2}\\
&&\lesssim\lf[\int_{I_k}|a_k(x)|^2\dmz(x)\r]^{1/2}\lf[m_{\lz}(I_k)\r]^{1/2}\\
&&\lesssim \lf[m_{\lz}(I_k)\r]^{-1/2}\lf[m_{\lz}(I_k)\r]^{1/2}\sim 1.
\end{eqnarray*}
For any $ x\notin 2I_{k}$, we see that
\begin{eqnarray}\label{osci-pois-atom}
\boscpz(a_k)(x)&&\le\sum_{j=1}^{\infty}\sup_{t_{j+1}\le\varepsilon_{j+1}<\varepsilon_{j}\le t_{j}}\lf|P_{\varepsilon_j}^{[\lz]}a_k(x)
-P_{\varepsilon_{j+1}}^{[\lz]}a_k(x)\r|\noz\\
&&=\sum_{j=1}^{\infty}\sup_{t_{j+1}\le\varepsilon_{j+1}<\varepsilon_{j}\le t_j}
\lf|\int_{\varepsilon_{j+1}}^{\varepsilon_j}\partial_{t}\plz(a_{k})(x)\,dt\r|\noz\\
&&\le \sum_{j=1}^{\infty}\dint_{t_{j+1}}^{t_j}\lf|\partial_{t}\plz (a_{k})(x)\r|\,dt\noz\\
&&\le\int_{0}^{\infty}\lf|\partial_{t}\plz(a_{k})(x)\r|\,dt.
\end{eqnarray}
Combining $\int_\rrp a_k(x)\dmz(x)=0$ and the mean value theorem yields
\begin{eqnarray}\label{upper bdd par-t atom}
\lf|\partial_{t}\plz a_k(x)\r|&&=\lf|\dint_{\rrp}\partial_{t}\plz(x,y)a_k(y)\,\dmz(y)\r|\noz\\
&&=\lf|\dint_{\rrp}\lf[\partial_{t}\plz (x,y)-\partial_{t}\plz(x,x_k)\r]a_k(y)\,\dmz(y)\r|\noz\\
&&\le \dint_{I_k}\lf|\partial_{y}\partial_{t}\plz (x,\xi)\r||y-x_k||a_k(y)|\,\dmz(y),
\end{eqnarray}
where $\xi:=sy+(1-s)x_k$ for $s\in(0, 1)$.

For any $y,\,\xi\in I_k$, $x\in\rrp\setminus 2I_k$, we have $|x-\xi|\sim|x-y|$.
  Thus, by  \eqref{upp bdd deriva meas}, we see that
\begin{eqnarray*}
{\rm E}_{k,2}&&\lesssim \dint_{\rrp \backslash 2I_k}\boscpz(a_k)(x)\,\dmz(x)\\
&&\lesssim\dint_{\rrp\backslash 2I_k}\dint_{0}^{\infty}\dint_{I_k}\lf|\partial_{y}\partial_{t}\plz(x,\xi)\r||y-x_k||a_k(y)|\,\dmz(y)\,dt\,\dmz(x)\\
&&\lesssim\dint_{\rrp\backslash 2I_k}\dint_{0}^{2r}\dint_{I_k}\dfrac{|y-x_k|}{m_\lz(I(x_k, |x-x_k|))}\frac1{(|x-x_k|+t)^2}|a_k(y)|\,\dmz(y)\,dt\,\dmz(x)\\
&&\qquad +\dint_{\rrp\backslash 2I_k}\dint_{2r}^{\infty}\dint_{I_k}\dfrac{|y-x_k|}{m_\lz(I(x_k, |x-x_k|))}\frac1{(|x-x_k|+t)^2}
|a_k(y)|\,\dmz(y)\,dt\,\dmz(x)\\
&&\lesssim r^2\dsum_{j=1}^{\infty}\dint_{I_k}\dint_{2^{j+1}I_k\backslash2^jI_k}\dfrac1{m_\lz(I(x_k, |x-x_k|))}\frac1{|x-x_k|^2}\,\dmz(x)|a_k(y)|\,\dmz(y)\\
&&\qquad+r\dsum_{j=1}^{\infty}\dint_{I_k}\dint_{2^{j+1}I_k\backslash2^jI_k}\dfrac1{m_\lz(I(x_k, |x-x_k|))}\frac1{|x-x_k|}\,\dmz(x)|a_k(y)|\,\dmz(y)\\
&&\lesssim r^2\dsum_{j=1}^{\infty}\dfrac{1}{(2^jr)^2}\frac{m_\lz(2^{j+1}I_k)}{m_\lz(2^{j}I_k)}
+r\dsum_{j=1}^{\infty}\dfrac{1}{2^jr}\frac{m_\lz(2^{j+1}I_k)}{m_\lz(2^{j}I_k)}\\
&&\lesssim 1.
\end{eqnarray*}
%In the case of $x> 2|x-y|$, by \eqref{par-yt2}, \eqref{osci-pois-atom} and \eqref{upper bdd par-t atom},
% we have $x\sim y\sim\xi$ and
%\begin{eqnarray*}
%{\rm E}_{k,2}&&\lesssim \dint_{\rrp \backslash 2I_k}\boscpz(a_k)(x)\,\dmz(x)\\
%&&\lesssim\dint_{\rrp\backslash 2I_k}\dint_{0}^{2r}\dint_{I_k}\dfrac{|y-x_0|}{(x\xi)^{\lz}(|x-y|+t)^{3}}|a_k(y)|\,\dmz(y)\,dt\,\dmz(x)\\
%&&\qquad +\dint_{\rrp\backslash 2I_k}\dint_{2r}^{\infty}\dint_{I_k}\dfrac{|y-x_0|}{(x\xi)^{\lz}(|x-y|+t)^{3}}|a_k(y)|\,\dmz(y)\,dt\,\dmz(x)\\
%&&\lesssim r^2\dsum_{j=1}^{\infty}\dfrac{1}{(2^jr)^3}\Big(\dint_{2^{j+1}I_k}x^{-2\lz}x^{2\lz}dx\Big)\dint_{I_k}|a_k(y)|\,\dmz(y)\\
%&&\qquad+ r\dsum_{j=1}^{\infty}\dfrac{1}{(2^jr)^2}\Big(\dint_{2^{j+1}I_k}x^{-2\lz}x^{2\lz}dx\Big)\dint_{I_k}|a_k(y)|\,\dmz(y)\\
%&&\lesssim 1.
%\end{eqnarray*}
Therefore,
\begin{equation*}
\dint_{\rrp}{\mathcal O\lf(P^{[\lz]}_\ast\r)} f(x)\dmz(x)\lesssim \sum_{k}|\alpha_k|.
\end{equation*}
This means
\begin{equation*}
\dint_{\rrp}{\mathcal O\lf(P^{[\lz]}_\ast\r)} f(x)\dmz(x)\lesssim \|f\|_{\hoz}.
\end{equation*}
This finishes the proof of Theorem \ref{t-H1 semigroup}.
\end{proof}

We next present the proof of Theorem \ref{t-char semigroup}.

\begin{proof}[Proof of Theorem \ref{t-char semigroup}]
First assume that $f\in \hoz$. Then $f\in\loz$. Moreover, by Theorem \ref{t-H1 semigroup}, we see that
$ V_\rho(P^{[\lz]}_{\ast})(f)\in\loz$ and
\begin{equation}\label{char-var upp bdd}
\|f\|_\loz+\lf\| V_\rho\lf(P^{[\lz]}_{\ast}\r)(f)\r\|_\loz\ls \|f\|_\hoz.
\end{equation}

Conversely, let $f\in\loz$ such that $V_\rho(P^{[\lz]}_{\ast})(f)\in\loz$. We first claim that for a. e. $x\in\rrp$,
\begin{equation}\label{MP}
\cm_{P^{[\lz]}}(f)(x)\le \mathcal V_\rho\lf(P^{[\lz]}_{\ast}\r)(f)(x)+|f(x)|,
\end{equation}
where
\begin{equation}\label{poi maxi funct}
\cm_{P^{[\lz]}}(f)(x):=\sup_{t>0}\lf|P^{[\lz]}_tf(x)\r|;
\end{equation}
see \cite{bdt}.  In fact, it suffices to verify that for any $t\in (0, \fz)$ and
a. e. $x\in\rrp$,
\begin{equation}\label{upper bdd poi vari}
\lf|P^{[\lz]}_tf(x)\r|\le\mathcal V_\rho\lf(P^{[\lz]}_{\ast}\r)(f)(x)+|f(x)|.
\end{equation}
To this end, recall that $P^{[\lz]}_tf\to f$ in $\loz$ as $t\to 0^+$; see \cite[p.\,362]{yy}.
Then for a. e. $x\in\rrp$, we have that $P^{[\lz]}_tf(x)\to f(x)$  as $t\to 0^+$.
Thus for fixed $t\in(0, \fz)$ and any $\varepsilon\in(0, \fz)$, there exists $t_0\in(0, t)$ such that
for any $\tilde t\in(0, t_0)$, $$\lf|P^{[\lz]}_{\tilde t}f(x)-f(x)\r|<\varepsilon.$$
Let $\tilde t\in(0, t_0)$ and $\{t_j\}_{j=1}^\fz\searrow0$ satisfying that  $t_1:=t$ and $t_2:=\tilde t$.
Then we conclude that
\begin{eqnarray*}
\lf|P^{[\lz]}_tf(x)\r|&\le& \lf|P^{[\lz]}_tf(x)-P^{[\lz]}_{\tilde t}f(x)\r|+\lf|P^{[\lz]}_{\tilde t}f(x)-f(x)\r|+|f(x)|\\
&<&\mathcal V_\rho\lf(P^{[\lz]}_{\ast}\r)(f)(x)+\varepsilon+|f(x)|.
\end{eqnarray*}
Since $\varepsilon$ is arbitrary,  we see that \eqref{upper bdd poi vari} holds.
This shows \eqref{MP}.

From \eqref{MP} and the assumption that $f,\,V_\rho(P^{[\lz]}_{\ast})(f)\in\loz$, it follows that $\cm_{P^{[\lz]}} f\in\loz$.
Recall that the space $\hoz$ is characterized in \cite[Theorem 1.7]{bdt} in terms of $\cm_{P^{[\lz]}}$ that
a function $f\in\loz$ is in $\hoz$ if and only $\cm_{P^{[\lz]}}f\in \loz$ and
$$\|f\|_\hoz\sim \|f\|_\loz+\lf\|\cm_{P^{[\lz]}}f\r\|_\loz.$$
This further implies that $f\in\hoz$ and
$$\|f\|_\hoz\ls \|f\|_\loz+\lf\| V_\rho\lf(P^{[\lz]}_{\ast}\r)(f)\r\|_\loz,$$
which together with \eqref{char-var upp bdd} completes the proof of Theorem \ref{t-char semigroup}.
\end{proof}

%%%%%%%%%%%%%%%%%%%%%%%%%%%%%%%%%%%%%%%%%%%%%%%%%%%%%%%%%%%%%%%%%%%%%%%%%%%%%%%%%%%%%%%%%%%%%%%%%%%%%%%%%%%%%%%%%%%%%%%%%%%%%%%%%%
%%%%%%%%%%%%%%%%%%%%%%%%%%%%%%%%%%%%%%%%%%%%%%%%%%%%%%%%%%%%%%%%%%%%%%%%%%%%%%%%%%%%%%%%%%%%%%%%%%%%%%%%%%%%%%%%%%%%%%%%%%%%%%%%%%

\section{The $BMO(\rrp,\dmz)$-type estimation}\label{s5}

 In this section, inspired by the methods in \cite{b14}, we apply Theorem \ref{t-bdd semigroup} and some properties of $BMO(\rrp,\dmz)$ space
 defined in \cite{cw77} to prove Theorem \ref{t-bmo semigroup}.

\begin{proof}[Proof of Theorem \ref{t-bmo semigroup}]
 Let $f\in \bmoz$, it suffices to  prove that for any interval $I:=I(x_0,r)$ with $x_0\in\rrp$ and $r>0$,
\begin{equation*}
\dfrac{1}{m_{\lz}(I)}\dint_{I}\lf|\boscpz (f)(x)-C_{B}\r|\dmz(x)\lesssim\|f\|_{\bmoz},
\end{equation*}
where $j_0\in\mathbb{N}$ such that $t_{j_0-1}\geq 8r>t_{j_0}$,
\begin{eqnarray*}
C_B:=&& \Bigg[\dsum_{j=1}^{j_0-2}\sup_{t_{j+1}\leq\varepsilon_{j+1}<\varepsilon_{j}\leq t_{j}}
\lf|P_{\varepsilon_{j+1}}^{[\lz]}f(x_0)-P_{\varepsilon_j}^{[\lz]}f(x_0)\r|^2\\
&&\qquad+\sup_{t_{j_0}\leq\varepsilon_{j_0}<\varepsilon_{j_0-1}\leq t_{j_0-1}}
\lf|P_{\varepsilon_{j_0}}^{[\lz]}(f,I)-P_{\varepsilon_{j_0-1}}^{[\lz]}(f,I)\r|^2 \Bigg]^{1/2},
\end{eqnarray*}
$f_{I,\,\lz}$ is as \eqref{f-i} and
$P_{t}^{[\lz]}(f,I):=P_{t}^{[\lz]}(f_{I,\,\lz})$ if $t<8r$ and $P_{t}^{[\lz]}(f,I):=P_{t}^{[\lz]}f(x_0)$ if $t\geq8r$.
Therefore,
\begin{eqnarray*}
&&\Big|\boscpz f(x)-C_B\Big|\\
&&\qquad=\Big|\Big(\dsum_{j=1}^{\infty}\dsup_{t_{j+1}\le\varepsilon_{j+1}<\varepsilon_j
\le t_j}\lf|P_{\varepsilon_{j+1}}^{[\lz]}f(x)-P_{\varepsilon_{j}}^{[\lz]}f(x)\r|^2\Big)^{1/2}-C_B\Big|\\
&&\qquad\le\Bigg|\Bigg[\dsum_{j=1}^{j_0-2}\dsup_{t_{j+1}\le\varepsilon_{j+1}<\varepsilon_j\le t_j}
\lf|P_{\varepsilon_{j+1}}^{[\lz]}f(x)-P_{\varepsilon_{j}}^{[\lz]}f(x)\r|^2\\
&&\qquad\qquad+
\dsup_{t_{j_0}\le\varepsilon_{j_0}<\varepsilon_{j_0-1}\le t_{j_0-1}}\lf|P_{\varepsilon_{j_0}}^{[\lz]}f(x)-P_{\varepsilon_{j_0-1}}^{[\lz]}f(x)\r|^2\\
&&\qquad\qquad\qquad+
\dsum_{j=j_0}^{\infty}\dsup_{t_{j+1}\le\varepsilon_{j+1}<\varepsilon_j\le t_j}
\lf|P_{\varepsilon_{j+1}}^{[\lz]}f(x)-P_{\varepsilon_{j}}^{[\lz]}f(x)\r|^2\Bigg]^{1/2}\\
&&\qquad\qquad-\Bigg[\dsum_{j=1}^{j_0-2}
\dsup_{t_{j+1}\le\varepsilon_{j+1}<\varepsilon_j\le t_j}\lf|P_{\varepsilon_{j+1}}^{[\lz]}f(x_0)-P_{\varepsilon_{j}}^{[\lz]}f(x_0)\r|^2\\
&&\qquad\qquad\qquad+
\dsup_{t_{j_0}\le\varepsilon_{j_0}<\varepsilon_{j_0-1}\le t_{j_0-1}}
\lf|P_{\varepsilon_{j_0}}^{[\lz]}(f,I)-P_{\varepsilon_{j_0-1}}^{[\lz]}(f,I)\r|^2\Bigg]^{1/2}\Bigg|\\
&&\qquad\le\lf[\dsum_{j=j_0}^{\infty}\dsup_{t_{j+1}\le\varepsilon_{j+1}<\varepsilon_j\le t_j}\lf|
P_{\varepsilon_{j+1}}^{[\lz]}f(x)-P_{\varepsilon_{j}}^{[\lz]}f(x)\r|^2\r]^{1/2}\\
&&\qquad\qquad+\lf[\dsum_{j=1}^{j_0-2}\dsup_{t_{j+1}\le\varepsilon_{j+1}<\varepsilon_j\le t_j}
\lf|\lf(P_{\varepsilon_{j+1}}^{[\lz]}-P_{\varepsilon_{j}}^{[\lz]}\r)(f(x)-f(x_0))\r|^2\r]^{1/2}\\
&&\qquad\qquad\qquad+\dsup_{0<t\le 8r}\lf|\plz(f-f_{I,\,\lz})(x)\r|+\dsup_{t>8r}\lf|\plz(f-f(x_0))(x)\r|\\
&&\qquad=:{\rm F}_1(x)+{\rm F}_2(x)+{\rm F}_3(x)+{\rm F}_4(x).
\end{eqnarray*}
It follows that
\begin{equation*}
\dfrac{1}{m_{\lz}(I)}\dint_{I}\lf|\boscpz f(x)-C_{B}\r|\dmz(x)\lesssim
\dsum_{i=1}^{4}\dfrac{1}{m_{\lz}(I)}\dint_{I}{\rm F}_i(x)dm_{\lz}(x)=:\dsum_{i=1}^{4}{\rm G}_i.
\end{equation*}
We write $$f=(f-f_{I,\,\lz})\chi_{2I}+(f-f_{I,\,\lz})\chi_{\rrp\backslash 2I}+f_{I,\,\lz}=:f_1+f_2+f_3.$$  Then, for ${\rm G}_1$ we have
\begin{eqnarray*}
{\rm G}_{1}&&\lesssim \dsum_{k=1}^{3}\dfrac{1}{m_{\lz}(I)}\dint_{I}\lf(\dsum_{j=j_0}^{\infty}\dsup_{t_{j+1}
\le\varepsilon_{j+1}<\varepsilon_j\le t_j}\lf|\lf(P_{\varepsilon_{j+1}}^{[\lambda]}-P_{\varepsilon_{j}}^{[\lz]}\r)f_k(x)\r|^2\r)^{1/2}dm_{\lz}(x)\\
&&=:{\rm G}_{11}+{\rm G}_{12}+{\rm G}_{13}.
\end{eqnarray*}

 From the conservation property $\plz(1)=1$ for all $t>0$, we deduce that ${\rm G}_{13}=0$.

By H\"older's inequality, John-Nirenberg's inequality (see \cite{cw77}) and the fact $\oscpz$ is bounded on $L^2(\rrp, \dmz)$, we obtain
\begin{eqnarray*}
{\rm G}_{11}&&\le \lf(\frac{1}{m_{\lz}(I)}\int_{I}\lf|\boscpz(f_1)(x)\r|^2\,\dmz(x)\r)^{1/2}\\
&&\lesssim\lf(\frac{1}{m_{\lz}(I)}\int_{2I}|f(x)-f_{I,\,\lz}|^2\,\dmz(x)\r)^{1/2}\\
&&\lesssim\|f\|_{\bmoz}.
\end{eqnarray*}

For ${\rm G}_{12}$, we write
\begin{eqnarray*}
{\rm G}_{12}&&\le\dfrac{1}{m_{\lz}(I)}\int_{I}\lf(\dsum_{j=j_0}^{\infty}\sup_{t_{j+1}
\le\varepsilon_{j+1}<\varepsilon_j\le t_j}\lf|\lf(P_{\varepsilon_j}^{[\lz]}-P_{\varepsilon_{j+1}}^{[\lz]}\r)f_2(x)\r|\r)\dmz(x)\\
&&\lesssim \frac{1}{m_{\lz}(I)}
\dint_{I}\int_{0}^{8r}\int_{0}^{\infty}\lf|\partial_{t}\plz(x,y)\r|\lf|f(y)-f_{I,\,\lz}\r|\chi_{\rrp \backslash 2I}\,y^{2\lz}\,dy\,dt\,\dmz(x)\\
&&\lesssim \sum_{k=1}^{\infty}\dfrac{1}{m_{\lz}(I)}\dint_{I}
\dint_{0}^{8r}\int_{2^{k+1}I\,\backslash\,2^kI}\lf|\partial_{t}P_{t}^{[\lz]}(x,y)\r|\lf|f(y)-f_{I,\,\lz}\r|y^{2\lz}\,dy\,dt\,dm_{\lz}(x).
\end{eqnarray*}

Similar to the estimate of \eqref{upp bdd deriva meas}, by Proposition \ref{p-upp bdd poisson} {\rm iii)}, we see that
for any $y\in\rrp\setminus2I$ and $x\in I$,
\begin{equation*}
\lf|\partial_{t}\plz(x,y)\r|\ls \frac1{m_\lz(I(x_0, |y-x_0|))}\frac1{|y-x_0|+t}.
\end{equation*}
From this and the fact that for any $k\in\nn$,
\begin{equation}\label{bmo regul}
|f_{I,\,\lz}-f_{2^kI,\,\lz}|\ls k\|f\|_\bmoz,
\end{equation}
it follows that
\begin{eqnarray*}
{\rm G}_{12}&&\lesssim\sum_{k=1}^{\fz}\frac{r}{m_{\lz}(I)}\dint_{I}\int_{2^{k+1}I\,\backslash\,2^kI}
\frac{|f(y)-f_{I,\,\lz}|}{m_\lz(I(x_0, |y-x_0|))}\frac1{|y-x_0|}y^{2\lz}\,dy\,\dmz(x)\\
&&\lesssim\dsum_{k=1}^{\fz}\frac{r}{2^kr}\frac1{m_\lz(2^kI)}\int_{2^{k+1}I}|f(y)-f_{I,\,\lz}|\,\dmz(y)\\
&&\lesssim\|f\|_{\bmoz}\sum_{k=1}^{\fz}\dfrac{k}{2^k}\\
&&\lesssim\|f\|_{\bmoz}.
\end{eqnarray*}
Combining the arguments of ${\rm G}_{11}$ and ${\rm G}_{12}$, we obtain ${\rm G}_{1}\lesssim\|f\|_{\bmoz}$.

For ${\rm G}_2$, by the mean value theorem, there exists $\xi:=sx_0+(1-s)x$ for some $s\in (0,\,1)$ such that
\begin{eqnarray*}
{\rm G}_{2}
&&\lesssim\frac{1}{m_{\lz}(I)}\int_{I}\int_{8r}^{\infty}\lf|\partial_{t}\plz(f-f_{I,\,\lz})(x)-\partial_{t}\plz(f-f_{I,\,\lz})(x_0)\r|\,dt\,\dmz(x)\\
&&\lesssim\frac{1}{m_{\lz}(I)}\int_{I}\int_{8r}^{\infty}\int_{0}^{\infty}\lf|
\partial_{t}\plz(x,y)-\partial_{t}\plz(x_0,y)\r|\lf|f(y)-f_{I,\,\lz}\r|y^{2\lz}\,dy\,dt\,\dmz(x)\\
&&\lesssim\frac{1}{m_{\lz}(I)}\int_{I}\int_{8r}^{\infty}\int_{0}^{\infty}\lf|
\partial_x\partial_{t}\plz(\xi,y)\r||x-x_0||f(y)-f_{I,\,\lz}|y^{2\lz}\,dy\,dt\,\dmz(x)\\
&&\lesssim\frac{1}{m_{\lz}(I)}\int_{I}\int_{8r}^{\infty}\int_{2I}\lf|\partial_x\partial_{t}\plz(\xi,y)\r||x-x_0||f(y)-f_{I,\,\lz}|y^{2\lz}\,dy\,dt\,\dmz(x)\\
&&\qquad+\frac{1}{m_{\lz}(I)}\int_{I}\int_{8r}^{\infty}\int_{\rrp\backslash 2I}
\lf|\partial_x\partial_{t}\plz(\xi,y)\r||x-x_0||f(y)-f_{I,\,\lz}|y^{2\lz}\,dy\,dt\,\dmz(x)\\
&&=:{\rm G}_{21}+{\rm G}_{22}.
\end{eqnarray*}

To estimate ${\rm G}_{21}$, we claim that for $t>8r$, $x,\,\xi\in I$ and $y\in 2I$,
\begin{equation}\label{upp bdd high deriva meas}
\lf|\partial_x\partial_{t}\plz(\xi,y)\r|\ls \frac1{m_\lz(I)}\frac1{t^2}.
\end{equation}
In fact, since by assumption, $x_0\ge r$, we then prove \eqref{upp bdd high deriva meas} by
 considering the following two cases.

 Case (i) $x_0\ge 4r$. In this case, we have $x\sim\xi\sim y\sim x_0$ and
 $m_\lz(I)\sim x_0^{2\lz}r.$
 By these facts and \eqref{p-p-t2}, we see that
  \begin{equation*}
\lf|\partial_x\partial_{t}\plz(\xi,y)\r|\ls \frac1{(y\xi)^\lz t^3}\ls \frac1{m_\lz(I)}\frac1{t^2}.
\end{equation*}

Case (ii) $r\le x_0< 4r$. In this case, $m_\lz(I)\sim r^{2\lz+1}$.  From this and \eqref{p-p-t1},
we deduce that
  \begin{equation*}
\lf|\partial_x\partial_{t}\plz(\xi,y)\r|\ls \frac1{t^{2\lz+3}}\ls \frac1{m_\lz(I)}\frac1{t^2}.
\end{equation*}
This implies the claim \eqref{upp bdd high deriva meas}.

By the claim \eqref{upp bdd high deriva meas} together with \eqref{bmo regul}, we conclude that
\begin{eqnarray*}
{\rm G}_{21}&&\lesssim \frac{1}{m_{\lz}(I)}\int_{I}\int_{8r}^{\infty}\int_{2I}
\frac{|x-x_0||f(y)-f_{I,\,\lz}|}{m_\lz(I)t^2}\,y^{2\lz}dy\,dt\,\dmz(x)\\
&&\lesssim\frac{r}{m_{\lz}(I)}\int_{2I}\frac{|f(y)-f_{I,\,\lz}|\,y^{2\lz}}{8r}\,dy\\
%&&\lesssim\frac{rm_{\lz}(2I)}{r^{2\lz+2}}\|f\|_{\bmoz}\\
&&\lesssim\|f\|_{\bmoz}.
\end{eqnarray*}

For $x,\,\xi\in I$ and $y\in\rrp\setminus2I$, we see  that
$$|\xi-y|\sim |x-y|\sim|x_0-y|,$$
and
\begin{equation*}
\lf|\partial_x\partial_{t}\plz(\xi,y)\r|\ls \frac1{m_\lz(I(x_0, |x_0-y|))}\frac1{(|x_0-y|+t)^2}.
\end{equation*}
 Applying \eqref{bmo regul} again, we obtain
\begin{eqnarray*}
{\rm G}_{22}&&\lesssim\sum_{k=1}^{\infty}\frac{r}{m_{\lz}(I)}\int_{I}\int_{8r}^{\infty}\int_{2^{k+1}I\backslash 2^kI}
\frac{|f(y)-f_{I,\,\lz}|}{m_\lz(I(x_0, |x_0-y|))}\frac1{(|x_0-y|+t)^2}y^{2\lz}\,dy\,\dmz(x)\\
&&\lesssim\sum_{k=1}^{\infty}\dfrac{r}{m_{\lz}(I)}\int_{I}\int_{2^{k+1}I\backslash 2^kI}\dfrac{|f(y)-f_{I,\,\lz}|}
{m_\lz(I(x_0, |x_0-y|))}\frac1{|x_0-y|+8r}y^{2\lz}\,dy\,dm_{\lambda}(x)\\
&&\lesssim\sum_{k=1}^{\fz}\dfrac{r}{2^kr}\frac1{m_\lz(2^kI)}\int_{2^{k+1}I}|f(y)-f_{I,\,\lz}|y^{2\lz}\,dy\\
%&&\qquad+\sum_{k=k_{0}+1}^{\infty}\dfrac{r}{m_{\lz}(I)(2^kr)^{2\lz+2}}\int_{I}\dmz(x)\int_{2^{k+1}I}|f(y)-f_{I,\,\lz}|y^{2\lz}\,dy\\
&&\lesssim\|f\|_{\bmoz}\sum_{k=1}^{\fz}\dfrac{k}{2^k}\\
&&\lesssim\|f\|_{\bmoz}.
\end{eqnarray*}
%In the case of $x\geq 2|x-y|$, we also have
%\begin{eqnarray*}
%{\rm G}_{22}&&\lesssim\sum_{k=1}^{k_0}\dfrac{r}{m_{\lambda}(I)}\int_{I}\int_{2^{k+1}I\backslash 2^kI}
%\dfrac{|f(y)-f_{I,\,\lz}|y^{2\lambda}}{(xy)^{\lambda}(|x-y|+8r)^{2}}\,dy\,dm_{\lambda}(x)\\
%&&\qquad+\sum_{k=k_{0}+1}^{\infty}\dfrac{r}{m_{\lambda}(I)}\int_{I}\int_{2^{k+1}I\backslash 2^kI}
%\dfrac{|f(y)-f_{I,\,\lz}|y^{2\lambda}}{(xy)^{\lambda}(|x-y|+8r)^{2}}\,dy\,dm_{\lambda}(x)\\
%&&\lesssim\sum_{k=1}^{k_0}\dfrac{rk|I|m_{\lambda}(2^{k+1}I)}{m_{\lambda}(I)(2^kr)^2}\|f\|_{\bmoz)}\\
%&&\qquad+\sum_{k=k_{0}+1}^{\infty}
%\dfrac{rkm_{\lambda}(2^{k+1}I)}{(2^kr)^{2\lambda+2}}\|f\|_{\bmoz}\\
%&&\lesssim\|f\|_{\bmoz)}.
%\end{eqnarray*}
Combining the estimates for ${\rm G}_{21}$ and ${\rm G}_{22}$, we have ${\rm G}_2\lesssim\|f\|_{\bmoz}$.

Now we estimate $\rm{G}_3$  by writing
\begin{eqnarray*}
{\rm G}_{3}&&\le\dfrac{1}{m_{\lambda}(I)}\int_{I}\sup_{0<t\le8r}\lf|P_{t}^{[\lambda]}((f-f_{I,\,\lz})\chi_{2I})(x)\r|\,dm_{\lambda}(x)\\
&&\qquad+\dfrac{1}{m_{\lambda}(I)}\int_{I}\sup_{0<t\le8r}\lf|P_{t}^{[\lambda]}((f-f_{I,\,\lz})\chi_{\mathbb{R}_{+}\backslash {2I}})(x)\r|\,dm_{\lambda}(x)\\
&&=:{\rm G}_{31}+{\rm G}_{32}.
\end{eqnarray*}
Recall that the Poisson maximal function $\cm_{P^{[\lz]}}f$  in \eqref{poi maxi funct}
is bounded on $\lpz$ for any $p\in(1, \fz)$(see \cite{bdt} or \cite[p.\,73]{st}).
Then by this, H\"older's inequality and John-Nirenberg's inequality, we conclude that
\begin{eqnarray*}
{\rm G}_{31}&&\lesssim \lf(\dfrac{1}{m_{\lambda}(I)}\int_{I}\lf[\cm_{P^{[\lz]}}((f-f_{I,\,\lz})\chi_{2I})(x)\r]^2dm_{\lambda}(x)\r)^{1/2}\\
&&\lesssim \lf(\dfrac{1}{m_{\lambda}(I)}\dint_{2I}|f(y)-f_{I,\,\lz}|^2\,dm_{\lambda}(x)\r)^{1/2}\\
&&\lesssim\|f\|_{\bmoz}.
\end{eqnarray*}
To estimate ${\rm G}_{31}$, observe that by Proposition \ref{p-upp bdd poisson} {\rm i)}, for $0<t\le 8r$ and $x\in I$,
\begin{eqnarray*}
\lf|P_{t}^{[\lambda]}((f-f_{I,\,\lz})\chi_{\mathbb{R}_{+}\backslash {2I}})(x)\r|
&&\lesssim \int_{\mathbb{R}_{+}\backslash {2I}}\lf|P_{t}^{[\lambda]}(x,y)\r||f(y)-f_{I,\,\lz}|\,y^{2\lambda}dy\\
&&\lesssim\sum_{k=1}^{\infty}\int_{2^{k+1}I\backslash {2^kI}}\lf|P_{t}^{[\lambda]}(x,y)\r||f(y)-f_{I,\,\lz}|\,y^{2\lambda}dy\\
&&\ls \sum_{k=1}^{\infty}\int_{2^{k+1}I\backslash {2^kI}}\frac{|f(y)-f_{I,\,\lz}|}{m_\lz(I(x_0, |x_0-y|))}\frac t{|x_0-y|+t}\,y^{2\lambda}dy.
\end{eqnarray*}
Therefore, from this and \eqref{bmo regul}, we deduce that
\begin{eqnarray*}
{\rm G}_{32}&&\lesssim \sum_{k=1}^{\fz}\frac{r}{2^kr}\frac{1}{m_{\lz}(2^kI)}\int_{2^{k+1}I}|f(y)-f_{I,\,\lz}|y^{2\lz}\,dy
\lesssim \|f\|_{\bmoz}.
\end{eqnarray*}
Therefore ${\rm G}_3\lesssim\|f\|_{\bmoz}.$

Since $\plz(1)=1$ for any $t>0$, by the mean value theorem, we see that
\begin{eqnarray*}
&&\sup_{t>8r}\lf|\plz f(x)-\plz f(x_0)\r|\\
&&\quad=\sup_{t>8r}\lf|\plz (f-f_{I,\,\lz})(x)-\plz (f-f_{I,\,\lz})(x_0)\r|\\
&&\quad\ls\sup_{t>8r}\int_{\rrp}\lf|\plz (x,y)-\plz(x_0,y)\r||f(y)-f_{I,\,\lz}|y^{2\lz}\,dy\\
&&\quad\ls\sup_{t>8r}\int_{\rrp}\lf|\partial_x\plz(\xi,y)\r||x-x_0||f(y)-f_{I,\,\lz}|y^{2\lz}\,dy,
 \end{eqnarray*}
 where $\xi\in I$.

Applying Proposition \ref{p-upp bdd poisson} {\rm ii)} and arguing as \eqref{upp bdd high deriva meas} and \eqref{upp bdd deriva meas},
we see that for $t>8r$, if $x,\,\xi\in I$ and $y\in 2I$,
\begin{equation*}
\lf|\partial_x\plz(\xi,y)\r|\ls \frac1{m_\lz(I)}\frac 1t;
\end{equation*}
and if $x,\,\xi\in I$ and $y\in \rrp\setminus 2I$,
\begin{equation*}
\lf|\partial_x\plz(\xi,y)\r|\ls \frac1{m_\lz(I(x_0, |x_0-y|))}\frac t{(|x_0-y|+t)^2}.
\end{equation*}
By this fact and \eqref{bmo regul}, we conclude that
 \begin{eqnarray*}
{\rm G}_4&&\lesssim\frac{1}{m_{\lz}(I)}\int_{I}\sup_{t>8r}\int_{2I}\frac{|f(y)-f_{I,\,\lz}|}{m_\lz(I)}\frac{|x-x_0|}{t}
y^{2\lz}\,dy\,\dmz(x)\\
&&\quad+\frac{1}{m_{\lz}(I)}\int_{I}\sup_{t>8r}\int_{\rrp\backslash2I}\frac{t|x-x_0|}{m_\lz(I(x_0, |x_0-y|))}\frac{|f(y)-f_{I,\,\lz}|}{(|x_0-y|+t)^2}
y^{2\lz}\,dy\,\dmz(x)\\
&&\lesssim\sup_{t>8r}\lf[\int_{2I}\frac{|f(y)-f_{I,\,\lz}|}{m_\lz(I)}\frac{r}{t}y^{2\lz}\,dy
+\int_{\rrp\backslash2I}\frac{|f(y)-f_{I,\,\lz}|}{m_\lz(I(x_0, |x_0-y|))}\frac{r}{|x_0-y|+t}y^{2\lz}\,dy\r]\\
&&\lesssim\frac1{m_\lz(I)}\int_{2I}|f(y)-f_{I,\,\lz}|y^{2\lz}\,dy
+r\int_{\rrp\backslash2I}\frac{|f(y)-f_{I,\,\lz}|}{m_\lz(I(x_0, |x_0-y|))|x_0-y|}y^{2\lz}\,dy\\
&&\ls\|f\|_\bmoz.
\end{eqnarray*}
Consequently, we get $\oscpz$ is bounded on $\bmoz$.
The proof of $\varpz$ is similar, we omit the details. This completes the proof of Theorem \ref{t-bmo semigroup}.
\end{proof}

\medskip

{\bf Acknowledgments}

This work is done during Dongyong Yang's visit to Macquarie University. He would like to
thank Professors Xuan Thinh Duong and Ji Li for their generous help.

{The first author is supported by the NNSF of China (Grant Nos. 11371295, 11471041) and the NSF of Fujian Province of China (No. 2015J01025).
The seconde author is supported by the NNSF of China (Grant No. 11571289) and the State Scholarship Fund of China (No. 201406315078).}

%\medskip

%Xuan Thinh Duong

%\smallskip
%
%Department of Mathematics, Macquarie University, NSW, 2109, Australia.
%
%\smallskip
%
%{\it E-mail}: \texttt{xuan.duong@mq.edu.au}
%
%\vspace{0.3cm}

Huoxiong Wu

School of Mathematical Sciences, Xiamen University, Xiamen 361005,  China
\smallskip

{\it E-mail}: \texttt{huoxwu@xmu.edu.cn}

\vspace{0.3cm}
Dongyong Yang

%\smallskip

School of Mathematical Sciences, Xiamen University, Xiamen 361005,  China

\smallskip

{\it E-mail}: \texttt{dyyang@xmu.edu.cn }

\vspace{0.3cm}

Jing Zhang

School of Mathematical Sciences, Xiamen University, Xiamen 361005,  China

School of Mathematics and Statistics, Yili Normal College, Yining Xinjiang 835000, China
\smallskip

{\it E-mail}: \texttt{zjmath66@126.com}


\begin{thebibliography}{99}

\bibitem{ak} K. F. Andersen and R. A. Kerman,
Weighted norm inequalities for generalized Hankel conjugate transformations,
{Studia Math.} {71} (1981/82), 15-26.

\vspace{-0.3cm}
\bibitem{Bou} J. Bougain,
Pointwise ergodic theorem for arithmatic sets,
{Inst. Hautes \'Etudes Sci. Publ. Math.} {69} (1989), 5-45.

\vspace{-0.3cm}
\bibitem{bcfr} J. J. Betancor, A. Chicco Ruiz, J. C. Fari\~{n}a and L. Rodr\'iguez-Mesa,
Maximal operators, Riesz transforms and Littlewood-Paley functions associated with Bessel operators on BMO,
{J. Math. Anal. Appl.} {363} (2010), 310-326.

\vspace{-0.3cm}
\bibitem{bct} J. J. Betancor, R. Crescimbeni, R., J. L. Torrea,
Oscillation and variation of the Laguerre heat and Poisson semigroups and Riesz transforms,
{Acta Math. Sci. Ser. B Engl. Ed.} {32} (2012), 907-928.

\vspace{-0.3cm}
\bibitem{bdt} J. J. Betancor, J. Dziuba\'nski and J. L. Torrea,
On Hardy spaces associated with Bessel operators,
{J. Anal. Math.} {107} (2009), 195-219.

\vspace{-0.3cm}
\bibitem{bfbmt} J. J. Betancor, J. C. Fari\~na, D. Buraczewski, T. Mart\'\i nez and J. L. Torrea,
Riesz transform related to Bessel operators,
{Proc. Roy. Soc. Edinburgh Sect. A} {137} (2007), 701-725.

\vspace{-0.3cm}
\bibitem{bfs} J. J. Betancor, J. C. Fari\~{n}a and A. Sanabria,
On Littlewood-Paley functions associated with Bessel operators,
{Glasg. Math. J.} {51} (2009), 55-70.

\vspace{-0.3cm}
\bibitem{bhnv} J. J. Betancor, E. Harboure, A. Nowak and B. Viviani,
Mapping properties of functional operators in harmonic analysis related to Bessel operators,
{Studia Math.} {197} (2010), 101-140.

\vspace{-0.3cm}
\bibitem{b14} T. A. Bui,
Boundedness of variation operators and oscillation operators for certain semigroups,
{Nonlinear Anal.} {106} (2014), 124-137.

\vspace{-0.3cm}
\bibitem{cjrw1} J. T. Campbell, R. L. Jones, K. Reinhold and M. Wierdl,
Oscillation and variation for the Hilbert transform,
{Duke Math. J.} {105} (2000), 59-83.

\vspace{-0.3cm}
\bibitem{cjrw2} J. T. Campbell, R. L. Jones, K. Reinhold and M. Wierdl,
 Oscillation and variation for singular integrals in higher dimensions,
 {Trans. Amer. Math. Soc.} {355} (2003), 2115-2137.

\vspace{-0.3cm}
\bibitem{cw71}  R. R. Coifman and G. Weiss,
Analyse Harmonique Non-commutative sur Certains Espaces Homog\`{e}nes,
Lecture Note in Mathematics, 242, Springer-Verlag, Berlin-New York (1971).

\vspace{-0.3cm}
\bibitem{cw77}  R. R. Coifman and G. Weiss,
Extensions of Hardy spaces and their use in analysis,
{Bull. Amer. Math. Soc.} {83} (1977), 569-645.

\vspace{-0.3cm}
\bibitem{cmtt} R. Crescimbeni, F. J. Mart\'in-Reyes, A. De La Torre and J. L. Torrea,
The $\rho$-variation of the Hermitian Riesz transform,
{ Acta Math. Sin. (Engl. Ser.)} {26} (2010), 1827-1838.

\vspace{-0.3cm}
\bibitem{cmtv} R. Crescimbeni, R. A. Mac\'ias, T. Men\'arguez, T and J. L. Torrea and B. Viviani,
The $¦Ñ$-variation as an operator between maximal operators and singular integrals,
{J. Evol. Equ.} {9} (2009), 81-102.

\vspace{-0.3cm}
\bibitem{dlwy} X. T. Duong, Ji Li, B. D. Wick and D. Yang,
Hardy space via factorization, and BMO space via commutators in the Bessel setting,
Submitted.

\vspace{-0.3cm}
\bibitem{gt} T. A. Gillespie and J. L. Torrea,
Dimension free estimates for the oscillation of Riesz transforms,
{Israel J. Math.} {141} (2004), 125-144.

\vspace{-0.3cm}
\bibitem{jkrw} R. L. Jones, R. Kaufman, J. M. Rosenblatt and M.  Wierdl,
Oscillation in ergodic theory,
Ergodic Theory Dynam. Systems 18 (1998), 889-935.

\vspace{-0.3cm}
\bibitem{jr} R. L. Jones and K. Reinhold,
Oscillation and variation inequalities for convolution powers,
{Ergodic Theory Dynam. Systems} {21} (2001), 1809-1829.

\vspace{-0.3cm}
\bibitem{jsw} R. L. Jones, A. Seeger and J. Wright,
Strong variational and jump inequalities in harmonic analysis,
{Amer. Math. Soc.} {360} (2008), 6711-6742.

\vspace{-0.3cm}
\bibitem{k78} R. A. Kerman,
Boundedness criteria for generalized Hankel conjugate transformations,
{Canad. J. Math.} {30} (1978), 147-153.

\vspace{-0.3cm}
\bibitem{LW} F. Liu and H. Wu,
A criterion on oscillation and variation for the commutators of singular integral operators,
{Forum Math.} {27} (2015), 77-97

\vspace{-0.3cm}
\bibitem{ms} B. Muckenhoupt and  E. M. Stein,
Classical expansions and their relation to conjugate harmonic functions,
{Trans. Amer. Math. Soc.} {118} (1965), 17-92.

\vspace{-0.3cm}
\bibitem{st}  E. M. Stein,
Topics in Harmonic Analysis Related to the Littlewood-Paley Theory,
Ann. of Math. Studies,vol. 63, Princeton Univ. Press, Princeton, NJ, (1970)

\vspace{-0.3cm}
\bibitem{v08} M. Villani,
Riesz transforms associated to Bessel operators,
{Illinois J. Math.} {52} (2008), 77-89.

\vspace{-0.3cm}
\bibitem{yy}  Da. Yang and Do. Yang,
Real-variable characterizations of Hardy spaces associated with Bessel operators,
{Anal. Appl. (Singap.)} {9 } (2011), 345-368.

\vspace{-0.3cm}
\bibitem{ZW} J. Zhang and H. Wu,
Oscillation and variation inequalities for the commutators of singular integrals with Lipschitz functions,
{J. Inequal. Appl.} 2015:214 (2015), 21 pp.




\end{thebibliography}
\end{document}